\documentclass[11pt]{amsart}
\usepackage{amssymb,amsfonts,amsmath,amscd,wasysym,amsthm}
\usepackage{verbatim} 
\usepackage[all]{xy} 
\usepackage{hyperref}

\theoremstyle{plain}

\newtheorem{theorem}{Theorem}[section] 
\newtheorem{lemma}[theorem]{Lemma}     
 
\newtheorem{corollary}[theorem]{Corollary}
\newtheorem{proposition}[theorem]{Proposition}

\newtheorem*{theorem*}{Theorem}
\newtheorem*{corollary*}{Corollary}

\theoremstyle{definition}

\newtheorem{remark}[theorem]{Remark}
\newtheorem{setting}[theorem]{Setting}
\newtheorem{example}[theorem]{Example}

\newcommand{\fm}{\mbox{$\mathfrak{m}$}}
\newcommand{\fp}{\mbox{$\mathfrak{p}$}}
\newcommand{\fq}{\mbox{$\mathfrak{q}$}}
\newcommand{\bn}{\mbox{$\mathbb{N}$}}
\newcommand{\bz}{\mbox{$\mathbb{Z}$}}

\newcommand{\mcm}{\mbox{$\mathcal{M}$}}

\newcommand{\mct}{\mbox{$\mathcal{T}$}}

\newcommand{\mcs}{\mbox{$\mathcal{S}$}}

\newcommand{\lp}{\mbox{$l_{\mathfrak{p}}$}}
\newcommand{\sigp}{\mbox{$\sigma_{\mathfrak{p}}$}}

\newcommand{\Min}{\mbox{${\rm Min}$}}

\newcommand{\length}{\mbox{${\rm length}$}}

\newcommand{\ch}{\mbox{${\rm char}$}}

\newcommand{\mult}{\mbox{${\rm mult}$}}
\newcommand{\embed}{\mbox{${\rm embed}$}}

\title{Non-complete intersection prime ideals in dimension $3$}

\author{{\sc Shiro Goto, Liam O'Carroll and Francesc Planas-Vilanova}}

\date{\today}

\begin{document}

\begin{abstract}
We describe prime ideals of height $2$ minimally generated by $3$
elements in a Gorenstein, Nagata local ring of Krull dimension $3$ and
multiplicity at most $3$. This subject is related to a conjecture of
Y.~Shimoda and to a long-standing problem of J.~Sally.
\end{abstract}

\maketitle


\section{Introduction}\label{introduction}

It is not known whether a Noetherian local ring, such that all its
prime ideals different from the maximal ideal are complete
intersections, has Krull dimension at most $2$. This problem was posed
by Y.~Shimoda and still remains unanswered in its full generality. In
fact, it is a partial version of a more general question of J.~Sally's,
namely, that the existence of a uniform bound on the minimal number of
generators of all its prime ideals is equivalent to the dimension of
the ring being at most $2$.

Note that, in the Shimoda problem, one may assume without loss of
generality that the local ring is Cohen-Macaulay and has dimension at
most $3$ (see \cite{gop} for more details, and particularly,
\cite[Remarks~2.2 and 2.4]{gop}).  Similarly, one may ask whether one
can display a prime ideal of height $2$ and minimally generated by at
least $3$ elements in a Cohen-Macaulay local ring of dimension $3$. By
a result due to M.~Miller \cite[Theorem~2.1]{miller}, under reasonably
general hypotheses, a local domain of dimension at least $4$ containing
a field possesses an abundance of prime ideals of height $2$ that are
not complete intersections.

The purpose of the paper is threefold: to generalise the results
obtained in the first part of \cite{gop}, to give simpler proofs, and
finally, to display a wide collection of examples to illustrate the
range of behaviour that occurs.

Let $(R,\fm,k)$ be a Cohen-Macaulay local ring, with $k$ infinite,
$\dim R=3$ and multiplicity $e(R)$. Let $(x,y,z)$ be a minimal
reduction of $\fm$. We ask for $k$ to be infinite just to ensure that
$\fm$ has a minimal reduction. If $R$ is regular local, we do not need
such an hypothesis, as $\fm$ is then its own minimal reduction.

Take $a=(a_{1},a_{2},a_{3})\in \bn_+^{3}$ and
$b=(b_{1},b_{2},b_{3})\in\bn_+^{3}$, where $\bn_+$ denotes the set of
positive integers; set $\bn=\{0\}\cup\bn_+$. Let $c=a+b$,
$c=(c_{1},c_{2},c_{3})$. Let $\mcm$ be the matrix
\begin{eqnarray*}
\mcm=\left(\begin{array}{ccc}x^{a_{1}}&y^{a_{2}}&z^{a_{3}}
  \\ y^{b_{2}}&z^{b_{3}}&x^{b_{1}}\end{array}\right),
\end{eqnarray*}
and $v_{1}=x^{c_{1}}-y^{b_{2}}z^{a_{3}}$,
$v_{2}=y^{c_{2}}-x^{a_{1}}z^{b_{3}}$ and
$D=z^{c_{3}}-x^{b_{1}}y^{a_{2}}$, the $2\times 2$ minors of $\mcm$ up
to a change of sign. Consider $I=(v_{1},v_{2},D)$, the determinantal
ideal generated by the $2\times 2$ minors of $\mcm$. Then $I$ is a
non-Gorenstein height-unmixed ideal of height two, minimally generated
by three elements (see \cite{op2}, where these ideals were called
Herzog-Northcott ideals, or HN ideals for short).

Throughout the paper we fix this notation, and $(R,\fm,k)$ and $I$
will be defined as above. Under additional assumptions on $R$, we will
study the minimal primary decomposition of $I$ and prove that either
$I$ itself is prime or else $I$ has a minimal prime which is not a
complete intersection, thus leading to the existence of prime ideals
of height $2$ and minimally generated by at least $3$ elements.

Set $m_1=c_{2}c_{3}-a_{2}b_{3}$, $m_2=c_{1}c_{3}-a_{3}b_{1}$,
$m_3=c_{1}c_{2}-a_{1}b_{2}$ and $m=(m_1,m_2,m_3)\in\bn_+^3$. Note that
each $m_i\geq 3$. We will always suppose that $m_1\leq m_2\leq m_3$
and that $\gcd(m_1,m_2,m_3)=1$ (changing $a$ to $(b_2,b_1,b_3)$ and
$b$ to $(a_2,a_1,a_3)$ changes $m$ to $(m_2,m_1,m_3)$; similarly,
changing $a$ to $(b_1,b_3,b_2)$ and $b$ to $(a_1,a_3,a_2)$ changes $m$
to $(m_1,m_3,m_2)$). Let $\mcs(I)=\langle m_1,m_2,m_3\rangle$ denote
the numerical semigroup generated by $m_1,m_2,m_3$ (see, e.g.,
\cite{rgs}).

Recall that a numerical semigroup $\mcs$ is a subset of $\bn$, closed
under addition, with $0\in\mcs$, and such that
$G(\mcs):=\bn\setminus\mcs$, the set of {\em gaps} of $\mcs$, is
finite. The cardinality of $G(\mcs)$ is denoted by $g(\mcs)$ and is
called the {\em genus} of $\mcs$. The {\em Frobenius number} $F(\mcs)$
of $\mcs$ is the greatest integer in $G(\mcs)$.  One can prove that
$g(\mcs)\geq (F(\mcs)+1)/2$. Moreover, $\mcs$ is {\em irreducible} if
it cannot be expressed as the intersection of two numerical semigroups
properly containing it, and $\mcs$ is {\em symmetric} if it is
irreducible and $F(\mcs)$ is odd. Alternatively, $\mcs$ is symmetric
if and only if $g(\mcs)=(F(\mcs)+1)/2$ (cf. \cite[Lemma~2.14 and
  Corollary~4.5]{rgs}).  Let $\{m_1<m_2<\ldots <m_r\}$ be the
(necessarily unique) minimal system of generators of a numerical
semigroup $\mcs$. The {\em multiplicity} of $\mcs$ is defined by the
expression $\mult(\mcs)=m_1$ and the {\em embedding dimension} of
$\mcs$ is defined by the expression $\embed(\mcs)=r$ (see
\cite[Theorem~2.7 and Proposition~2.10]{rgs}). Every numerical
semigroup of emdedding dimension two is symmetric
(\cite[Corollary~4.7]{rgs}).

For any other unexplained notation, we refer to \cite{bh} or
\cite{hs}. Our main result is as follows. Note that a minimal prime
over $I$ is necessarily of height two.

\begin{theorem*} 
Let $(R,\fm,k)$ be a Gorenstein, Nagata local ring, with $k$ infinite,
and $\dim R=3$. Let $(x,y,z)$ be a minimal reduction of $\fm$. Let
$I=(x^{c_{1}}-y^{b_{2}}z^{a_{3}},y^{c_{2}}-x^{a_{1}}z^{b_{3}},z^{c_{3}}-x^{b_{1}}y^{a_{2}})$.
Suppose that $\mcs(I)=\langle m_1,m_2,m_3\rangle$ is not contained in
any symmetric semigroup $\mcs$ with $\mult(\mcs)=m_1$. If $e(R)\leq 3$,
then either $I$ is prime, or else there exists a minimal prime $\fp$
over $I$ such that $\fp$ is not a complete intersection.
\end{theorem*}

This result generalises \cite[Proposition~2.8]{gop}, since on the one
hand, a complete Noetherian local ring $R$ is Nagata (see
\cite[Chapter~12, \S~31, Corollary~2]{matsumura2}) and, on the other
hand, we do not need the ring to be a domain or contain the residue
field. As a consequence, it generalises the main result in \cite{gop},
since the hypotheses of \cite[Theorem~2.3]{gop} imply that $R$ is
Gorenstein and Nagata. In other words, we obtain the following
result. Recall that a Noetherian local ring is {\em Shimoda} if every
prime ideal in the punctured spectrum is of the principal class.

\begin{corollary*} 
Let $(R,\fm,k)$ be a Shimoda ring of dimension $d\geq 2$. Then $d=2$
provided that either $R$ is regular, or else $R$ is Gorenstein,
Nagata, $k$ is infinite and $e(R)\leq 3$.
\end{corollary*}

We finish the paper with examples that show that each one of the
particular cases arising in the main theorem can occur. 

\section{Preliminary results}\label{preliminary}

We start by substantiating some remarks on the multiplicity of $R$ and
$R/I$.

\begin{remark}\label{reductionsI}
We first observe that $R/I$ is a one-dimensional Cohen-Macaulay
ring. Next we remark that $xR/I$ is a minimal reduction of $\fm
R/I$. Indeed, and with an obvious abuse of notation, in $R/I$ one has
the following equalities: $x^{c_{1}}=y^{b_{2}}z^{a_{3}}$,
$y^{c_{2}}=x^{a_{1}}z^{b_{3}}$ and
$z^{c_{3}}=x^{b_{1}}y^{a_{2}}$. Then it is easy to check that
$y^{c_2c_3}=x^{m_2}y^{a_2b_3}$. Since $y$ is not a zero divisor in
$R/I$, $y^{m_1}=x^{m_2}$ and $x^{m_2}$ belongs to $(xR/I)^{m_1}$ since
$m_1\leq m_2$. Therefore $y\in\overline{xR/I}$, the integral closure
of the ideal $xR/I$. Analogously, one can check that
$z^{m_1}=x^{m_3}\in (xR/I)^{m_1}$, so $z\in \overline{xR/I}$. Hence
$xR/I$ is a reduction of $(x,y,z)R/I$. Since $(x,y,z)R/I$ is a
reduction of $\fm R/I$, then $xR/I$ is a reduction of $\fm R/I$. Since
$\dim R/I=1$, $xR/I$ is a minimal reduction of $\fm R/I$. Observe also
that $x+I$ forms a regular sequence in $R/I$. In particular,
\begin{eqnarray*}
&&e(R/I)=e_{R/I}(\fm R/I;R/I)=e_{R/I}(xR/I;R/I)=\\&&
=e_{R/I}(x+I;R/I)=\length_{R/I}((R/I)/(x+I)R/I)=\length_R(R/(xR+I)).
\end{eqnarray*}

Analogously, if $\fp$ is a minimal prime over $I$, then $xR/\fp$ is a
minimal reduction of $\fm R/\fp$ and
$e(R/\fp)=e_{R/\mathfrak{p}}(xR/\fp;R/\fp)= \length_R(R/(xR+\fp))$.
\end{remark}

\begin{lemma}\label{eRIism1e}
$e(R/I)=m_1e(R)$.
\end{lemma}
\begin{proof}
By Remark~\ref{reductionsI}, $e(R/I)=\length_R(R/(xR+I))$, where
$xR+I=(x,y^{c_2},y^{b_2}z^{a_3},z^{c_3})$. With $S=R/xR$, note that
$R/(xR+I)\cong S/(y^{c_2},y^{b_2}z^{a_3},z^{c_3})S$. In the
two-dimensional Cohen-Macaulay local ring $S$, and with an obvious
abuse of notation, $y,z$ is a regular sequence and a system of
parameters. By \cite[Lemma~2.9]{gop},
$\length_{R}(R/(xR+I))=\length_{S}(S/(y^{c_2},y^{b_2}z^{a_3},z^{c_3})S)=
m_1\length_S(S/(y,z)S)$.  Since $S/(y,z)S\cong R/(x,y,z)$ and $x,y,z$
is a minimal reduction of $\fm$, then
$\length_S(S/(y,z)S)=\length_R(R/(x,y,z))=e_R(x,y,z;R)=e_R(\fm;R)=e(R)$.
\end{proof}

We now fix some more notations.

\begin{setting}\label{DandV}
For a minimal prime $\fp$ over $I$, let $D=R/\fp$, which is a
one-dimensional Noetherian local domain with maximal ideal $\fm_D$,
say. Let $V=\overline{D}$ be the integral closure of $D$ in its
quotient field; then $V$ is a Dedekind domain by the Krull-Akizuki
Theorem. If $Q$ is a maximal ideal of $V$, then $V_Q$ is a DVR. Let
$\fm_{V_Q}=QV_Q$ denote its maximal ideal, $k_{V_Q}$ its residue field
and $\nu_{Q}$ its valuation. If $V$ is local, let $\fm_V$ denote its
maximal ideal, $k_V$ its residue field and $\nu$ its valuation. If $V$
is local and $k=k_V$ under the natural identification, one says that
$k$ is {\em residually rational}. If $R$ is a Nagata ring, then $V$ is
a finitely generated $D$-module.
\end{setting}

\begin{proposition}\label{eD}
Let $\fp$ be a minimal prime over $I$. The following hold.
\begin{itemize}
\item[$(a)$] For any $Q$, there exists $\eta=\eta(Q)\in\bn_{+}$ with
  that $(\nu_Q(x),\nu_Q(y),\nu_Q(z))=\eta (m_1,m_2,m_3)$.
\item[$(b)$] $e(D)>1$.
\end{itemize}
Suppose that, in addition, $R$ is Nagata. The following hold.
\begin{itemize}
\item[$(c)$] $e(D)=m_1\sigp$, where
  $\sigp=\sum_{Q}\eta(Q)[k_{V_Q}:k]$.
\item[$(d)$] $e(D)=m_1$ if and only if $V$ is a DVR, $\eta=1$ and $k$
  is residually rational. 
\item[$(e)$] Moreover, if $e(D)=m_1$, then $D$ is analytically
  irreducible.
\end{itemize}
\end{proposition}
\begin{proof}
Any maximal ideal $Q$ of $V$ contracts to $\fm_D$ through the natural
inclusion $D\subseteq V$, so $\fm_DV\subseteq Q$. Therefore, in $V_Q$,
on applying $\nu_Q$ to the equalities $x^{c_{1}}=y^{b_{2}}z^{a_{3}}$,
$y^{c_{2}}=x^{a_{1}}z^{b_{3}}$ and $z^{c_{3}}=x^{b_{1}}y^{a_{2}}$, one
gets $(\nu_Q(x),\nu_Q(y),\nu_Q(z))=\eta (m_1,m_2,m_3)$, for some
non-zero rational number $\eta=\eta(Q)$ depending on $Q$ and $\fp$
(see \cite[Remark~4.4]{op2}). Write $\eta=u/v$, with
$u,v\in\bn_{+}$. Then $v(\nu_Q(x),\nu_Q(y),\nu_Q(z))=u(m_1,m_2,m_3)$
and, on taking the greatest common divisor, one has
$v\gcd(\nu_Q(x),\nu_Q(y),\nu_Q(z))=u \gcd(m_1,m_2,m_3)=u$. So
$\gcd(\nu_Q(x),\nu_Q(y),\nu_Q(z))=u/v$ and $\eta=u/v\in\bn_+$.

By Remark~\ref{reductionsI}, $e(D)=\length_R(R/(xR+\fp))$. If
$\length_R(R/(xR+\fp))=1$, then $\fm=xR+\fp$ and $R/\fp$ is a DVR with
valuation $\nu$, say, and uniformizing parameter $x$ (by abuse of
notation), so $\nu(x)=1$. Applying $(a)$, this forces $m_1=1$, which
is in contradiction to $m_1\geq 3$. This proves $(b)$.

Suppose that $R$ is Nagata. Applying Remark~\ref{reductionsI} and
\cite[Theorem~11.2.7]{hs},
\begin{eqnarray*}
e(D)=e_D(xD;D)=\sum_Q e_{V_Q}(xV_Q;V_Q)[k_{V_Q}:k],
\end{eqnarray*}
where $Q$ runs over the maximal ideals of $V$. Applying $(a)$,
$(\nu_Q(x),\nu_Q(y),\nu_Q(z))=\eta (m_1,m_2,m_3)$, for some
$\eta=\eta(Q)\in\bn_+$. In particular,
$e_{V_Q}(xV_Q;V_Q)=\length(V_Q/xV_Q)=\eta(Q)m_1$. Therefore
\begin{eqnarray*}\label{mult}
e(D)=e_D(xD;D)=\sum_Q e_{V_Q}(xV_Q;V_Q)[k_{V_Q}:k]=m_1\sum_Q
\eta(Q)[k_{V_Q}:k]=m_1\sigp,
\end{eqnarray*}
where $\sigp=\sum_Q \eta(Q)[k_{V_Q}:k]$. Hence $e(D)=m_1\sigp\geq
m_1$, and $e(D)=m_1$ is equivalent to $\sigp=1$. Moreover, $\sigp=1$
is equivalent to $V$ being local and so a DVR with valuation $\nu$
say, $\eta=1$ (i.e., $\nu(x)=m_1$, $\nu(y)=m_2$ and $\nu(z)=m_3$) and
$[k_V:k]=1$. Furthermore, in this case, $D$ is analytically
irreducible since the $\fm_D$-adic completion of $D$ can be seen as a
subring in the $\fm_V$-adic completion of $V$, which is a DVR, whence
a domain. (For the converse statement, see \cite[p.~486,
  Section~1]{matsuoka}.)
\end{proof}

Given a numerical semigroup $\mcs$ with Frobenius number $F(\mcs)$,
set $N(\mcs)=\{ s\in\mcs\mid s<F(\mcs)\}$ and $n(\mcs)=|N(\mcs)|$ its
cardinality. Note that $g(\mcs)+n(\mcs)=F(\mcs)+1$. Since $g(\mcs)\geq
(F(\mcs)+1)/2$, it follows that $(F(\mcs)+1)\geq 2n(\mcs)$ (see
\cite[just before Proposition~2.26]{rgs}).

\begin{proposition}\label{notGor}
Suppose that $R$ is Nagata and that $\mcs(I)$ is not contained in any
symmetric semigroup $\mcs$ with $\mult(\mcs)=m_1$. Let $\fp$ be a
minimal prime over $I$ such that $e(D)=m_1$. Then $D$ is not
Gorenstein.
\end{proposition}
\begin{proof}
Observe that $D$ cannot be a DVR since $m_1\geq 3$. Hence the
conductor $(D:_DV)\subseteq\fm_D$, where $V=\overline{D}$. By
Remark~\ref{reductionsI}, $xD$ is a minimal reduction of $\fm_D$, so
$\overline{xD}=\fm_D$ (see \cite[Corollary~1.2.5]{hs}). By
\cite[Theorem 6.8.1]{hs}, $\fm_D\subseteq
\fm_DV=\overline{(xD)}V=xV$. By Proposition~\ref{eD},~$(d)$, $V$ is a
DVR with uniformizing parameter $t$ and valuation $\nu$, say, and
$\nu(x)=m_1$, $\nu(y)=m_2$ and $\nu(z)=m_3$. In particular, the
numerical semigroup $\langle m_1,m_2,m_3\rangle$ is contained in the
numerical semigroup $\nu(D)$. Moreover, $xV=t^{m_1}V$ and
$(D:_DV)\subseteq\fm_D\subseteq \fm_DV=xV=t^{m_1}V$.  Therefore,
$\fm_D\subseteq t^{m_1}V$ and
\begin{eqnarray}\label{inclusions}
\mcs(I)=\langle m_1,m_2,m_3\rangle\subseteq \nu(D)\subseteq
\{0\}\cup\{ n\in\bz, n\geq m_1\}.
\end{eqnarray}
Thus, $\nu(D)$ is a numerical semigroup containing $\mcs(I)$ and of
multiplicity $\mult(\nu(D))=m_1$. By hypothesis, $g(\nu(D))>
(F(\nu(D))+1)/2$ or, equivalently, $(F(\nu(D))+1)> 2n(\nu(D))$.

By Proposition~\ref{eD},~$(d)$, $k$ is residually rational. Applying
\cite[Remark, Page 40]{bdf} (see also \cite[Proposition~1]{matsuoka}),
we obtain $\length_V(V/(D:_DV))=F(\nu(D))+1$ and
$\length_D(D/(D:_DV))=n(\nu(D))$.  In particular,
$\length_V(V/(D:_DV))>2\length_D(D/(D:_DV))$ and, by
\cite[Theorem~12.2.2]{hs}, $D$ cannot be Gorenstein.
\end{proof}

Now let $\fp$ run through $\Min(R/I)$, the set of minimal primes over
$I$. Let $n_I$ be the cardinality of $\Min(R/I)$. For each minimal
prime $\fp$ over $I$, set $\lp=\length_{R_{\mathfrak{p}}}
(R_{\mathfrak{p}}/I_{\mathfrak{p}})$. Recall from
Proposition~\ref{eD}, $(c)$, that $e(R/\fp)=m_1\sigp$.

\begin{corollary}\label{cases}
Suppose that $R$ is Nagata. Then $e(R)=\sum
_{\mathfrak{p}}\sigp\lp$. In particular, $n_I\leq e(R)$. Moreover, for
small values of $e(R)$, we have the following possibilities.
\begin{itemize}
\item[$(a)$] If $e(R)=1$, then $n_I=1$, $\Min(R/I)=\{\fp\}$,
  $(\sigp,\lp)=(1,1)$ and $I=\fp$ is prime with $e(R/\fp)=m_1$.
\item[$(b)$] Suppose that $e(R)=2$. Then either
\begin{itemize}
\item[$(b.1)$] $n_I=1$, $\Min(R/I)=\{\fp\}$, $(\sigp,\lp)=(2,1)$ and
  $I=\fp$ is prime with $e(R/\fp)=2m_1$, or
\item[$(b.2)$] $n_I=1$, $\Min(R/I)=\{\fp\}$, $(\sigp,\lp)=(1,2)$ and
  $I$ is $\fp$-primary with $e(R/\fp)=m_1$, or
\item[$(b.3)$] $n_I=2$, $\Min(R/I)=\{\fp_1,\fp_2\}$,
  $(\sigma_{\mathfrak{p}_i},l_{\mathfrak{p}_i})=(1,1)$ for $i=1,2$,
  and $I=\fp_1\cap \fp_2$ with each $e(R/\fp_i)=m_1$.
\end{itemize}
\item[$(c)$] Suppose that $e(R)=3$. Then either
\begin{itemize}
\item[$(c.1)$] $n_I=1$, $\Min(R/I)=\{\fp\}$, $(\sigp,\lp)=(3,1)$ and
  $I=\fp$ is prime with $e(R/\fp)=3m_1$, or
\item[$(c.2)$] $n_I=1$, $\Min(R/I)=\{\fp\}$, $(\sigp,\lp)=(1,3)$ and
  $I$ is $\fp$-primary with $e(R/\fp)=m_1$, or
\item[$(c.3)$] $n_I=2$, $\Min(R/I)=\{\fp_1,\fp_2\}$,
  $(\sigma_{\mathfrak{p}_1},l_{\mathfrak{p}_1})=(1,2)$,
  $(\sigma_{\mathfrak{p}_2},l_{\mathfrak{p}_2})=(1,1)$ and
  $I=\fq_1\cap \fp_2$ with $\fq_1$ a $\fp_1$-primary ideal and
  each $e(R/\fp_i)=m_1$, or
\item[$(c.4)$] $n_I=2$, $\Min(R/I)=\{\fp_1,\fp_2\}$,
  $(\sigma_{\mathfrak{p}_1},l_{\mathfrak{p}_1})=(2,1)$,
  $(\sigma_{\mathfrak{p}_2},l_{\mathfrak{p}_2})=(1,1)$ and
  $I=\fp_1\cap \fp_2$ with $e(R/\fp_1)=2m_1$ and $e(R/\fp_2)=m_1$, or
\item[$(c.5)$] $n_I=3$, $\Min(R/I)=\{\fp_1,\fp_2,\fp_3\}$,
  $(\sigma_{\mathfrak{p}_i},l_{\mathfrak{p}_i})=(1,1)$ for $i=1,2,3$,
  and $I=\fp_1\cap \fp_2\cap \fp_3$ with each $e(R/\fp_i)=m_1$.
\end{itemize}
\end{itemize}
In particular, if $e(R)\leq 3$, then either $I$ is prime, or else
there exists a minimal prime $\fp$ over $I$ such that $e(D)=m_1$, with
$D$ not Gorenstein, provided that $\mcs(I)$ is not contained in any
symmetric semigroup $\mcs$ with $\mult(\mcs)=m_1$.
\end{corollary}
\begin{proof}
By Lemma~\ref{eRIism1e}, the Associativity Law of Multiplicities and
Proposition~\ref{eD},
\begin{eqnarray*}
m_1e(R)=e(R/I)=e_{R/I}(xR/I;R/I)=\sum
_{\mathfrak{p}}e_{R/\mathfrak{p}}(xR/\fp;R/\fp)
\length_{R_{\mathfrak{p}}}(R_{\mathfrak{p}}/I_{\mathfrak{p}})=m_1\sum
_{\mathfrak{p}}\sigp\lp.
\end{eqnarray*}
Thus $e(R)=\sum _{\mathfrak{p}}\sigp\lp$. In particular, $n_I\leq
e(R)$. If $e(R)=1$, one deduces that $I$ has a unique minimal prime
$\fp$ and that, for such $\fp$,
$\length_{R_{\mathfrak{p}}}(R_{\mathfrak{p}}/I_{\mathfrak{p}})=1$, so
$I=\fp$. (See \cite[Proposition~2.6]{gop}; recall that, for a
Cohen-Macaulay local ring $R$, $e(R)=1$ is equivalent to $R$ being a
regular local ring: cf. \cite[Theorem~40.6 and Corollary~25.3]{nagata}
or \cite[Exercise~11.8]{hs}.) The rest of the assertions follow
analogously. One finishes by applying Propositions~\ref{eD}, $(c)$, and
\ref{notGor}.
\end{proof}

\begin{example}\label{3,4,5}
Let $(R,\fm,k)$ be a Cohen-Macaulay, Nagata local ring, with $k$
infinite, and $\dim R=3$. Let $(x,y,z)$ be a minimal reduction of
$\fm$. Let $I=(x^3-yz,y^2-xz,z^2-x^2y)$. If $e(R)\leq 3$, then either
$I$ is prime, or else there exists a minimal prime $\fp$ over $I$ such
that $D$ is not Gorenstein with $e(D)=3$, these two cases overlapping
precisely when $e(R)=1$. (See Section~\ref{examples} to note that each
of the two possibilities can occur.) Moreover, in the latter case, $D$
is an almost Gorenstein ring and the canonical ideal $\omega_D$ of $D$
is minimally generated by two elements.
\end{example}
\begin{proof}
Note that $\mcs(I)=\langle 3,4,5\rangle$ is not contained in any
symmetric semigroup $\mcs$ with $\mult(\mcs)=3$. By
Corollary~\ref{cases}, either $I$ is prime, or else 
\begin{eqnarray}\label{2ndcase}
\mbox{ there exists a minimal prime $\fp$ over $I$ such that $e(D)=3$
  and $D$ is not Gorenstein.}
\end{eqnarray}
In the latter case (\ref{2ndcase}), by Proposition~\ref{eD}, such $D$
is analytically irreducible.

Suppose that (\ref{2ndcase}) holds. Then the chain of inclusions
(\ref{inclusions}) in Proposition~\ref{notGor} must be a chain of
equalities, so $\nu(D)=\langle 3,4,5\rangle$. Note that
$F(\nu(D))=2$. So $\length_V(V/(D:_DV))=F(\nu(D))+1=3$. Since $V$ is a
DVR, it follows that $(D:_DV)=t^3V$, so $(D:_DV)=\fm_D=xV=t^3V$. In
particular, $\fm_DV\subseteq D$ and $D$ is an almost Gorenstein ring
(see \cite[Corollary~3.12]{gmp}; see also \cite{bf}).

Since $(D:_DV)=\fm_D$, then $\length_D(D/(D:_DV))=1$. Furthermore, $D$
analytically irreducible implies that $D$ admits a canonical ideal
$\omega_D$ (see, e.g., \cite[Proposition~2.7]{gmp}). By
\cite[Theorem~12.2.3]{hs}, $3=\length_D(V/(D:_DV))\geq
2\length_D(D/(D:_DV))+\mu(\omega_D)-1=1+\mu(\omega_D)$, where $\mu$
stands for ``minimal number of generators''.  Therefore,
$\mu(\omega_D)\leq 2$. Since $D$ is not Gorenstein, this forces
$\mu(\omega_D)=2$ (alternatively, this follows also from the
definition of ``almost Gorenstein'' in \cite[page~418]{bf}).
\end{proof}

\section{Main theorem}\label{mainresult}

Now, we can state and prove the main result of the paper. We keep the
same notations.

\begin{theorem}\label{mainT}
Let $(R,\fm,k)$ be a Gorenstein, Nagata local ring, with $k$ infinite,
and $\dim R=3$. Let $(x,y,z)$ be a minimal reduction of $\fm$. Let
$I=(x^{c_{1}}-y^{b_{2}}z^{a_{3}},y^{c_{2}}-x^{a_{1}}z^{b_{3}},z^{c_{3}}-x^{b_{1}}y^{a_{2}})$.
Suppose that $\mcs(I)=\langle m_1,m_2,m_3\rangle$ is not contained in
any symmetric semigroup $\mcs$ with $\mult(\mcs)=m_1$. If $e(R)\leq 3$,
then either $I$ is prime, or else there exists a minimal prime $\fp$
over $I$ such that $\fp$ is not a complete intersection.
\end{theorem}
\begin{proof}
By Corollary~\ref{cases}, either $I$ is prime, or else there exists a
minimal prime $\fp$ over $I$ such that $D$ not Gorenstein. In
particular, since $R$ is Gorenstein, $\fp$ cannot be a complete
intersection (\cite[Proposition~3.1.19]{bh}).
\end{proof}

The following result clarifies the hypothesis ``$\mcs(I)$ not
contained in any symmetric semigroup $\mcs$ with $\mult(\mcs)=m_1$''.
Let $\mct$ be the numerical semigroup $\mct=\langle
m_1,m_2,m_3\rangle$ with $3\leq m_1\leq m_2\leq m_3$ and
$\gcd(m_1,m_2,m_3)=1$. In particular $\mult(\mct)=m_1$ and
$\embed(\mct)\leq 3$. If $\embed(\mct)=2$, then $\mct$ is symmetric
(see \cite[Corollary~4.5]{rgs}). Therefore, in order to fulfill the
hypotheses of Proposition~\ref{notGor} and Theorem~\ref{mainT}, we can
suppose that $\embed(\mct)=3$. Hence $m_1<m_2<m_3$.

\begin{proposition}\label{semigroups}
Let $\mct=\langle m_1,m_2,m_3\rangle$ be a numerical semigroup with
$3\leq m_1<m_2<m_3$ and $\gcd(m_1,m_2,m_3)=1$. Suppose that
$\embed(\mct)=3$. Let $\Delta=\{ \langle 3,4,5\rangle, \langle
3,5,7\rangle, \langle 4,5,7\rangle, \langle 4,7,9\rangle\}$.  Then
$\mct$ is not contained in any symmetric semigroup $\mcs$ with
$\mult(\mcs)=m_1$ if and only if $\mct\in \Delta$.
\end{proposition}
\begin{proof}
The ``if'' implication is a simple check. We now prove the ``only if''
implication. Take $\mct=\langle m_1,m_2,m_3\rangle$ and suppose that
$\mct\not\in\Delta$. Let us show that $\mct$ is contained in a
symmetric semigroup $\mcs$ with $\mult(\mcs)=m_1$.

Observe that since $\embed(\mct)=3$, then $m_3\not\in\langle
m_1,m_2\rangle$ and $m_3>m_2$. For the sake of simplicity, set
$BG(m_1,m_2)=G( \langle m_1,m_2\rangle )\cap \{ m\in\bn_+, m>m_2\}$,
where $G(\langle m_1,m_2\rangle)$ is the set of gaps of $\langle
m_1,m_2\rangle$ ($BG$ standing for ``big gaps''). Thus $m_3\in
BG(m_1,m_2)$.

Suppose that $m_1=3$ and $m_2=4$. Then $m_3\in BG(m_1,m_2)=\{5\}$, in
contradiction to $\mct\not\in\Delta$. Analogously, if $m_1=3$ and
$m_2=5$, then $m_3\in BG(m_1,m_2)=\{7\}$, in contradiction to
$\mct\not\in\Delta$. Therefore, if $m_1=3$, then $m_2\geq 6$ and
$\mct\subseteq \langle 3,4\rangle=\{0,3,4,6,\mapsto\}$, which is
symmetric.

Suppose that $m_1=4$. Set $\mcs_1=\langle
4,5,6\rangle=\{0,4,5,6,8,\mapsto\}$, $\mcs_2=\langle
4,6,7\rangle=\{0,4,6,7,8,10,\mapsto\}$, which are symmetric. Let us
prove that either $\mct\subseteq \mcs_1$, or else $\mct\subseteq
\mcs_2$.  Indeed, if $m_2=5$, then $m_3\in
BG(m_1,m_2)=\{6,7,11\}$. Since $\mct\not\in\Delta$, $m_3\in\{6,11\}$
and $\mct\subseteq \mcs_1$. Suppose that $m_2=6$. If $m_3=9$,
$\mct\subseteq\mcs_1$. If $m_3\neq 9$, $\mct\subseteq \mcs_2$. Suppose
that $m_2=7$. Since $\mct\not\in\Delta$, then $m_3\neq 9$ and
$\mct\subseteq \mcs_2$. If $m_1=4$ and $m_2\geq 8$, then
$\mct\subseteq \mcs_1$.

Suppose that $m_1\geq 5$. Take $\mcs_1=\langle m_1,m_1+1,\ldots
,2m_1-2\rangle$ and $\mcs_2=\langle m_1,m_1+2,\ldots
,2m_1-1\rangle$. One can check that $F(\mcs_1)=2m_1-1$,
$F(\mcs_2)=2m_1+1$, and that $\mcs_1$ and $\mcs_2$ are symmetric.

If $m_2\geq 2m_1$, then $\mct\subseteq \mcs_1$. Suppose that
$m_2=2m_1-1$. If $m_3\neq 2m_1+1$, then $\mct\subseteq \mcs_2$. If
$m_3=2m_1+1$, then $\mct\subseteq \langle m_1,2m_1-1,2m_1+1,\ldots
,3m_1-4,3m_1-2\rangle$, which is symmetric (with Frobenius number
$4m_1-3$).

Suppose that $m_2\leq 2m_1-2$. If $m_3\neq 2m_1-1$, then
$\mct\subseteq \mcs_1$. If $m_3=2m_1-1$ and $m_2\neq m_1+1$, then
$\mct\subseteq \mcs_2$. Finally, if $m_2=m_1+1$ and $m_3=2m_1-1$,
$\mct\subseteq \langle m_1,m_1+1,m_1+4,\ldots ,2m_1-1\rangle$, which
is symmetric (with Frobenius number $2m_1+3$).
\end{proof}

\begin{remark}
Recall that $a=(a_1,a_2,a_3)\in\bn_+^3$, $b=(b_1,b_2,b_3)\in\bn_+^3$
and $c=a+b$. Moreover
$m_1=c_{2}c_{3}-a_{2}b_{3}=a_2a_3+a_3b_2+b_2b_3$,
$m_2=c_{1}c_{3}-a_{3}b_{1}=a_1a_3+a_1b_3+b_1b_3$, and
$m_3=c_{1}c_{2}-a_{1}b_{2}=a_1a_2+a_2b_1+b_1b_2$. It is easy to check
that the following four matrices
\begin{eqnarray*}
\mcm_1=\left(\begin{array}{llc}x^{}&y^{}&z^{}
  \\ y^{}&z^{}&x^{2}\end{array}\right)\mbox{ , }
\mcm_2=\left(\begin{array}{llc}x^{}&y^{}&z^{}
  \\ y^{}&z^{}&x^{3}\end{array}\right)\mbox{ , }
\mcm_3=\left(\begin{array}{llc}x^{2}&y^{2}&z^{}
  \\ y^{}&z^{}&x^{}\end{array}\right)\mbox{ , }
\mcm_4=\left(\begin{array}{llc}x^{3}&y^{2}&z^{}
  \\ y^{}&z^{}&x^{}\end{array}\right),
\end{eqnarray*}
give rise to the corresponding ideals of $2\times 2$ minors
\begin{eqnarray*}
&&I_1=(x^3-yz,y^2-xz,z^2-x^2y)\mbox{ , }
I_2=(x^4-yz,y^2-xz,z^2-x^3y)\mbox{ , }\\
&&I_3=(x^3-yz,y^3-x^2z,z^2-xy^2)\mbox{ , }
I_4=(x^4-yz,y^3-x^3z,z^2-xy^2)\mbox{ , }
\end{eqnarray*}
with $S(I_1)=\langle 3,4,5\rangle$, $S(I_2)=\langle 3,5,7\rangle$,
$S(I_3)=\langle 4,5,7\rangle$ and $S(I_4)=\langle 4,7,9\rangle$, the
four semigroups appearing in the set $\Delta$.

In fact, these are the only examples with prescribed semigroup in
$\Delta$. Indeed, if $m_1=3$, then $a_2$, $a_3$, $b_2$ and $b_3$ must
be equal to $1$. Substituting in the expressions of $m_2$ and $m_3$
leads to a $2\times 2$ system with solution $a_1=(1/3)(2m_2-m_3)$ and
$b_1=(1/3)(2m_3-m_2)$. If $m_2=4$ and $m_3=5$, then $a_1=1$ and
$b_1=2$. If $m_2=5$ and $m_2=7$, then $a_1=1$ and $b_1=3$.

If $m_1=4$, this forces either $a_2=2$ and $a_3$, $b_2$ and $b_3$
equal to $1$, or else $b_3=2$ and $a_2$, $a_3$ and $b_2$ equal to
$1$. If $a_2=2$, substituting in the expressions of $m_2$ and $m_3$,
one gets a $2\times 2$ system with solution $a_1=(1/4)(3m_2-m_3)$ and
$b_1=(1/2)(m_3-m_2)$. If $m_2=5$ and $m_3=7$, then $a_1=2$ and
$b_1=1$.  If $m_2=7$ and $m_3=9$, then $a_1=3$ and $b_1=1$. Finally,
if $b_3=2$, substituting in the expressions of $m_2$ and $m_3$, one
gets a $2\times 2$ system with solution $a_1=(1/2)(m_2-m_3)$ and
$b_1=(1/4)(3m_3-m_2)$. However $m_2<m_3$ would force $a_1<0$, which
makes no sense.
\end{remark}

\section{Examples}\label{examples}

Our next purpose is to display examples of each one of the cases in
Corollary~\ref{cases}. First we fix the notations for the rest of the
paper.

\begin{setting}\label{generalsetting}
Let $k$ be a field and let $X$, $Y$, $Z$, $W$, $t$ be indeterminates
over $k$. Set $A=k[X,Y,Z]$, $\fm_A=(X,Y,Z)A$ and
$S=A_{\mathfrak{m}_A}$, the localization of $A$ in $\fm_A$. Call
$\fm_S$ the maximal ideal of $S$. Take $a$, $b$, $c$ and $m\in\bn^3_+$
as in in Section~\ref{introduction} and suppose that $m_1<m_2<m_3$ and
$\gcd(m_1,m_2,m_3)=1$. Let
$J=(X^{c_{1}}-Y^{b_{2}}Z^{a_{3}},Y^{c_{2}}-X^{a_{1}}Z^{b_{3}},
Z^{c_{3}}-X^{b_{1}}Y^{a_{2}})A\subset \fm_A$. By
\cite[Theorem~7.8]{op2}, $J$ is a prime ideal of $A$. In fact,
$J=\ker(\varphi_m:A\to k[t])$, where $\varphi_m$ sends $X$, $Y$ and
$Z$ to $t^{m_1}$, $t^{m_2}$ and $t^{m_3}$, respectively. In
particular, $JS$ is a prime ideal of $S$.

Set $B=A[W]=K[X,Y,Z,W]$, $\fm_B=(X,Y,Z,W)B$ and
$T=B_{\mathfrak{m}_B}$, the localization of $B$ in $\fm_B$. Call
$\fm_T$ the maximal ideal of $T$. By abuse of notation, we consider
elements of $A$ to be elements of $B$ and elements of $S$ to be
elements of $T$. Let $n\geq 1$, $g_1,\ldots ,g_n$, with
$g_i\in\fm_A^i$ and $f=W^n+g_1W^{n-1}+\ldots
+g_n\in(WB+\fm_AB)^n$. Note that $JB+fB\subset \fm_B$.

We now specify our model for the ring $R$ and our model for the ideal
$I$ that will exemplify the results considered in the paper,
particularly as regards Theorem~\ref{mainT} and
Corollary~\ref{cases}. Take $R=T/fT$, the factor ring of $T$ modulo
$f$. Let $\fm_R$ denote the maximal ideal of $R$. Let lower-case
letters $x$, $y$, $z$, $w$ denote the corresponding image elements in
$R$. Thus $\fm_R=(x,y,z,w)R$ and clearly $(R,\fm_R,k)$ is a
Gorenstein, Nagata local ring of dimension $\dim R=3$. Since $w$ is
integral over the ideal $(x,y,z)R$, then $(x,y,z)R$ is a minimal
reduction of $\fm_R$. Now take $I=JR=(x^{c_{1}}-y^{b_{2}}z^{a_{3}},
y^{c_{2}}-x^{a_{1}}z^{b_{3}}, z^{c_{3}}-x^{b_{1}}y^{a_{2}})R$. Clearly
$e(R)=n$, by a standard result (see \cite[Example~11.2.8]{hs}, say);
alternatively, by calculation, since $x,y,z$ is a regular sequence in
$R$, then $e(R)=e_R((x,y,z)R;R)=\length_R(R/(x,y,z)R)$, so, setting
$T^{\prime}=T/(X,Y,Z)T$,
\begin{eqnarray*}
e(R)=\length_T(T/(X,Y,Z,f)T)=\length_{T^{\prime}}(T^{\prime}/W^nT^{\prime})=n.
\end{eqnarray*}
\end{setting}

Let us study the minimal primary decomposition of $I$ for different
particular choices of the element $f$. We start with the cases in
Corollary~\ref{cases} in which $I$ is prime.

\begin{example}\label{caseab1c1} 
{\sc Cases $(a)$, $(b.1)$ and $(c.1)$}. 
\begin{itemize}
\item[$(i)$] In Setting~\ref{generalsetting}, take
  $f=W^n-X^{n-1}Y$. When $n\in\{1,2\}$, take $m=(m_1,m_2,m_3)$ in $\{
  (3,4,5),(4,5,7),(4,7,9)\}$; when $n=3$, take $m$ in
  $\{(3,4,5),(3,5,7),(4,5,7)\}$. Note that for each choice of $n$ and
  $m$, $\gcd(nm_1,nm_2,nm_3,(n-1)m_1+m_2)=1$. Let $P=\ker(\psi)$,
  where $\psi:B\to k[t]$ sends $X$, $Y$, $Z$ and $W$ to $t^{nm_1}$,
  $t^{nm_2}$, $t^{nm_3}$ and $t^{(n-1)m_1+m_2}$, respectively. Then
  $P=JB+fB$. In particular, $JT+fT$ is a prime ideal of $T$. Thus
  $e(R)=n$ and $I$ is a prime ideal of $R$.
\item[$(ii)$] Take $f=W^n-X^{n-1}Z$, $n=3$, in
  Setting~\ref{generalsetting} and take $m$ in $\{
  (3,4,5),(3,5,7),(4,7,9)\}$. Note that for each choice of $m$,
  $\gcd(nm_1,nm_2,nm_3,(n-1)m_1+m_3)=1$. Let $P=\ker(\psi)$, where
  $\psi:B\to k[t]$ sends $X$, $Y$, $Z$ and $W$ to $t^{nm_1}$,
  $t^{nm_2}$, $t^{nm_3}$ and $t^{(n-1)m_1+m_3}$, respectively. Then
  $P=JB+fB$. In particular, $JT+fT$ is a prime ideal of $T$. Thus
  $e(R)=n$ and $I$ is a prime ideal of $R$.
\end{itemize}
\end{example}
\begin{proof}
$(i)$ It suffices to adapt the proofs of \cite[Remark~7.2, Lemma~7.5
    and Theorem~7.8]{op2} to the ring $B$ and the ideal $JB+fB$, with
  the variables $X$, $Y$, $Z$ and $W$ being given weights $nm_1$,
  $nm_2$, $nm_3$ and $(n-1)m_1+m_2$, respectively. In this regard,
  note that $JB+fB$ is unmixed, since $J(B/fB)$ is unmixed by
  \cite[Proposition~2.2, $(b)$]{op2}.

$(ii)$ This follows similarly, with the variables $X$, $Y$, $Z$ and
  $W$ now given weights $nm_1$, $nm_2$, $nm_3$ and $(n-1)m_1+m_3$,
  respectively.
\end{proof}

An example covering Case $(b.1)$ when $m=(3,5,7)$ is shown in
Example~\ref{case357}. Before proceeding, we need some prior
observations.

\begin{remark}\label{varphi}
Take $g\in \fm_A$. Then $g$ defines a surjective evaluation map
$\varphi_g:B\to A$, where $\varphi_g$ fixes $k$, $X$, $Y$ and $Z$ and
sends $W$ to $g$. Note that if $p\in B\setminus\fm_B$, then
$p(0,0,0,g(0,0,0))=p(0,0,0,0)\neq 0$, so $\varphi_g(p)\in
A\setminus\fm_A$, and if $q\in A\setminus\fm_A$, then $q\in
B\setminus\fm_B$ and $\varphi_g(q)=q$. In particular, $\varphi_g$ can
be extended to a morphism, $\varphi_g:T\to S$, say, that is a
retraction of the natural inclusion $S\subset T$.
\end{remark}

\begin{lemma}\label{kernel}
Let $g\in\fm_A$. Then $\ker(\varphi_g:B\to A)=(W-g)B$ and
$\ker(\varphi_g:T\to S)=(W-g)T$. In particular, $JB+(W-g)B$ is a prime
ideal of height $3$ in $B$ and $JT+(W-g)T$ is a prime ideal of height
$3$ in $T$.
\end{lemma}
\begin{proof} That $\ker(\varphi_g:B\to A)=(W-g)B$ follows easily 
from the appropriate Division Algorithm. The second assertion follows
since localisation is a flat functor, so kernels are preserved. In
particular, since $JA$ is a prime of height $2$ in $A$ and
$\varphi_g(JB)=JA$, then (via $\varphi_g^{-1}$) $JB+(W-g)B/(W-g)B$ is
a prime of height $2$ in $B/(W-g)B$, so $JB+(W-g)B$ is a prime ideal
of height $3$ in $B$ because $W-g$ is prime in $B$. Analogously,
$JT+(W-g)T$ is a prime ideal of height $3$ in $T$.
\end{proof}

Next we note some elementary facts about lifting a minimal primary
decomposition over an ideal. We shall use these facts below without
explicit mention.

\begin{remark}\label{generalfacts}
Let $L,K$ be ideals in a Noetherian ring $C$ such that $L\supseteq
K$. For $i=1,\ldots,r$, consider ideals $Q_i$ and $P_i$ with
$P_i\supseteq Q_i\supseteq L$ such that in $C/K$ we have the minimal
primary decomposition $L/K=\cap_iQ_i/K$, where each $P_i/K$ is a prime
ideal and $Q_i/K$ is $P_i/K$-primary. Then in $C$, $L=\cap_iQ_i$ is a
minimal primary decomposition, and for $i=1,\ldots,r$, each $P_i$ is a
prime ideal and $Q_i$ is $P_i$-primary. In particular, if $L/K$ is an
unmixed ideal in $C/K$, then $L$ is an unmixed ideal in $C$.
\end{remark}
\begin{proof} Note that for each $i$, $C/P_{i}\simeq (C/K)/(P_{i}/K)$, 
so $C/P_i$ is a domain. Moreover, $C/Q_{i}\simeq (C/K)/(Q_{i}/K)$, so
in $C/Q_i$ each divisor of zero is nilpotent. The remainder of the
assertions follow from the basic theory of ideals in factor rings.
\end{proof}

\begin{example}\label{caseb2c2} {\sc Cases $(b.2)$ and $(c.2)$}. 
Take $f=W^n$, $n\geq 1$, in Setting~\ref{generalsetting}. Then
$P=JB+WB$ is a prime ideal of $B$ contained in $\fm_B$. Set
$\fp=PR$. Then $e(R)=n$, $\Min(R/I)=\{\fp\}$, $(\sigp,\lp)=(1,n)$ and
$I$ is $\fp$-primary with $e(R/\fp)=m_1$.
\end{example}
\begin{proof}
By Lemma~\ref{kernel}, $P$ is a prime ideal of height $3$. Since
$I=JR$ is unmixed (see \cite[Proposition~2.2]{op2}), it follows easily
that $PT$ is the unique prime minimal over $JT+fT$.

Set $U=T_{PT}$ (the localisation of $T$ at the prime $PT$). Then
$V=U/IU$ is a one-dimensional local domain with maximal ideal
generated by the image of $W$ in $V$. Hence $V$ is a DVR. It is
immediate that $V/W^nV$ is of length $n$ (as $V$-module). By
definition, this length is the local length of $JT+fT$ at $PT$. Since
$R=T/fT$, we deduce that $\lp$, the local length of $I$ at its unique
minimal prime $\fp=PR$, equals $n$.
\end{proof}

\begin{example}\label{caseb3c3} {\sc Cases $(b.3)$ and $(c.3)$}. 
Take $f=W^{n-1}(W-X)$, $n\geq 2$, in
Setting~\ref{generalsetting}. Then $P_1=JB+WB$ and $P_2=JB+(W-X)B$ are
prime ideals of $B$ contained in $\fm_B$. Set $\fp_i=P_iT$,
$i=1,2$. Then $e(R)=n$, $\Min(R/\fp)=\{\fp_1,\fp_2\}$,
$(\sigma_{\mathfrak{p}_1},l_{\mathfrak{p}_1})=(1,n-1)$,
$(\sigma_{\mathfrak{p}_2},l_{\mathfrak{p}_2})=(1,1)$ and
$I=\fq_1\cap\fp_2$ is a minimal primary decomposition with $\fq_1$ a
$\fp_1$-primary ideal and $e(R/\fp_i)=m_1$.
\end{example}
\begin{proof}
By Lemma~\ref{kernel}, $P_1=JB+WB$ and $P_2=JB+(W-X)B$ are prime
ideals of $B$ contained in $\fm_B$. Since $I=JR$ is unmixed, it
follows that $P_i$ are the only minimal primes above $JT+fT$. Note
that $P_1$ and $P_2$ are distinct, since
$\varphi_{0}(P_1)\neq\varphi_{0}(P_2),$ as is easily seen from the
fact that $X\not\in J$. In particular, $W\not\in P_2$ and $W-X\not\in
P_1$. A simple localization argument shows that $JT+fT=P_1\cap P_2$.
\end{proof}

\begin{example}\label{casec4} {\sc Case $(c.4)$}. 
Take $f=(W^{n-1}-X^{n-2}Y)(W-X)$, $n=3$, in
Setting~\ref{generalsetting}. As in Example~\ref{caseab1c1}, take
$m=(m_1,m_2,m_3)$ in $\{(3,4,5),(4,5,7),(4,7,9)\}$. Then we claim that
$P_1=JB+(W^{n-1}-X^{n-2}Y)B$ and $P_2=JB+(W-X)B$ are prime ideals of
$B$ contained in $\fm_B$. The latter holds by Lemma~\ref{kernel}. To
see the former, it suffices to repeat the argument of
Example~\ref{caseab1c1}~$(i)$ only now having $\psi$ send $X$, $Y$,
$Z$ and $W$ to $t^{(n-1)m_1}$, $t^{(n-1)m_2}$, $t^{(n-1)m_3}$ and
$t^{(n-2)m_1+m_2}$, respectively. Note that in each case
$\gcd((n-1)m_1,(n-1)m_2,(n-1)m_3,(n-2)m_1+m_2)=1$. Set $\fp_i=P_iT$,
$i=1,2$. Then $e(R)=n$, $\Min(R/\fp)=\{\fp_1,\fp_2\}$
$(\sigma_{\mathfrak{p}_1},l_{\mathfrak{p}_1})=(n-1,1)$,
$(\sigma_{\mathfrak{p}_2},l_{\mathfrak{p}_2})=(1,1)$ and
$I=\fp_1\cap\fp_2$ is a minimal primary decomposition with
$e(R/\fp_1)=(n-1)m_1$ and $e(R/\fp_2)=m_1$. (We leave the details to
the reader.)
\end{example}

\begin{example}\label{casec5} {\sc Case $(c.5)$}. 
Take $f=W^{n-2}(W-X)(W-Y)$, $n\geq 3$, in
Setting~\ref{generalsetting}. By Lemma~\ref{kernel}, $P_1=JB+WB$,
$P_2=JB+(W-X)B$ and $P_3=JB+(W-Y)B$ are prime ideals of $B$ contained
in $\fm_B$. Set $\fp_i=P_iT$, $i=1,2,3$. Then $e(R)=n$,
$\Min(R/\fp)=\{\fp_1,\fp_2,\fp_3\}$,
$(\sigma_{\mathfrak{p}_1},l_{\mathfrak{p}_1})=(1,n-2)$,
$(\sigma_{\mathfrak{p}_i},l_{\mathfrak{p}_i})=(1,1)$, for $i=2,3$, and
$I=\fq_1\cap\fp_2\cap\fp_3$ is a minimal primary decomposition with
$\fq_1$ a $\fp_1$-primary ideal and $e(R/\fp_i)=m_1$, for
$i=1,2,3$. (Details are left to the reader.)
\end{example}

\begin{remark}\label{case-domain}
We can even find examples with $f$ a prime element in $B$, hence $R$ a
domain, with some restrictions on the base field $k$. Note that in
Example~\ref{caseab1c1}, for $n=3$, $f=W^3-X^2Y$ is irreducible in
$B$. Indeed, suppose that $f$ has a factor of the form $W-g$, for some
$g\in A$. Then $\varphi_g(f)=0$, so $g^3=X^2Y$. Since $X$ and $Y$ are
irreducible elements in the UFD $A$, this yields a contradiction.

For the cases $(b.3)$ and $(c.3)$, as in Example~\ref{caseb3c3}, and
with $m=(3,4,5)$ and $n=2$, take $f=W^2-XZ$, which is irreducible in
$B$, by an analogous argument. If $\ch(k)\neq 2$, then
$I=(JR+(w-y)R)\cap (JR+(w+y)R)$ is a minimal primary
decomposition. 

For the case $(c.4)$, as in Example~\ref{casec4}, and with $m=(4,5,7)$
and $n=3$, take $f=W^3-X^2Z$, which analogously is irreducible in
$B$. 

If $k$ is separable and does not contain a cube root of unity
different from $1$, then one can show, by a rather lengthy and
technical argument not given here, that $I=(JR+(w-y)R)\cap
(JR+(w^2+yw+y^2)R)$ is a minimal primary decomposition. (Hint: Extend
the base field from $k$ to $k[\lambda]$, where $\lambda$ is a
primitive cube root of unity. Use the properties of integral and
faithfully flat extensions, together with the Cohen-Seidenberg Theorem
\cite[Theorem~5, pp.~33-34]{matsumura2}, particularly
\cite[Theorem~5,~vi)]{matsumura2}.)

For the case $(c.5)$, as in Example~\ref{casec5}, with $m=(4,5,7)$ and
$n=3$ and $f=W^3-X^2Z$ as above, and if $k$ contains a cube root of
unity $\lambda\neq 1$ (and so three distinct cube roots of unity
$1,\lambda,\lambda^2$), then $I=\cap _{j=0}^{2}(JR+(w-\lambda^jy)R)$
is a minimal primary decomposition.

Note that in these examples, for instance when $f=W^2-XZ$, while $R$
is a domain, it is not a UFD, since $w$, $x$ and $z$ are prime
elements in $R$ yet $w^2=xz$. Here, $(x,w)R$ is a non-principal prime
ideal of height $1$ (and $R$ is not Shimoda, see
Section~\ref{introduction}).
\end{remark}

\begin{example}\label{case357} 
{\sc Case $(b.1)$, $m=(3,5,7)$}. Let $k$ be a field of characteristic
different from $2$, not containing a square root of $-1$. Let
$f=W^2+XZ$. Then $JB+fB=JB+(W^2+Y^2)B$. An analogue of the Hint in
Remark~\ref{case-domain} above shows that $JB+fB$ is a prime
ideal. Thus $e(R)=2$ and $I$ is a prime ideal.
\end{example}

\begin{remark}
The examples above prove that all the cases in Corollary~\ref{cases}
and in the main theorem can occur. They also suggest that the
condition $e(R)\leq 3$ is not strictly necessary. However the proof of
Theorem~\ref{mainT} strongly relies on applying the Associative Law of
Multiplicities for small values of $e(R)$. It seems clear then that
radically different techniques will be needed in order to extend
Theorem~\ref{mainT} (still in dimension $3$) to the case of higher, or
indeed arbitrary, multiplicities.
\end{remark}

\section*{Acknowledgment} 
The third author is partially supported by grant
MTM2010-20279-C02-01. Part of this work was carried out while the
third author was visiting the International Centre for Mathematics at
Edinburgh. He wants to thank the Centre as well as the Department of
Mathematics of the University of Edinburgh for their hospitality,
friendliness and good atmosphere for working.

{\small
}

\vspace{0.2cm}
{\footnotesize 

\noindent {\sc Department of Mathematics, School of Science and
  Technology, Meiji University}, \\ Tama, Kawasaki, KANAG 214,
Japan. {\em E-mail address}: goto@math.meiji.ac.jp

 \vspace{0.1cm}

\noindent {\sc Maxwell Institute for Mathematical Sciences, School of
  Mathematics, University of Edinburgh}, \\ EH9 3JZ, Edinburgh, Scotland.
          {\em E-mail address}: L.O'Carroll@ed.ac.uk

\vspace{0.1cm}

\noindent {\sc Departament de Matem\`atica Aplicada~1, Universitat
  Polit\`ecnica de Catalunya}, \\ Diagonal 647, ETSEIB, 08028 Barcelona,
Catalunya. {\em E-mail address}: francesc.planas@upc.edu }
\end{document}